\documentclass[12pt]{article}
\usepackage{amsmath,amsthm,amssymb}
\usepackage{hyperref, pdfpages}
\usepackage{graphics}

\def\Inv{\operatorname{Inv}}
\def\Meas{\operatorname{Meas}}

\newtheorem{theorem}{\noindent Theorem}
\newtheorem{lemma}{\noindent Lemma}
\newtheorem{definition}{\noindent Definition}
\newtheorem{corollary}{\noindent Corollary}

\newtheorem{statement}{\noindent Proposition}
\newtheorem{remark}{\noindent Remark}
\newtheorem{problem}{\noindent Problem}
\newtheorem{conclusion}{\noindent Conclusion}
\urldef\vershik\url{vershik@pdmi.ras.ru}

\title{The problem of describing central measures on the path spaces of graded graphs}

\author{A.~M.~Vershik\thanks{%
St.~Petersburg Department of Steklov Institute of Mathematics and St.~Petersburg State University.
E-mail: \vershik. Supported by the Russian Science Foundation grant 14-11-00581.}}

\date{07.08.14}

\begin{document}

  \maketitle

\begin{abstract}
We suggest a new method of describing invariant measures on Markov compacta and path spaces of graphs, and thus of describing characters of some groups and traces of AF-algebras. The method relies on properties of filtrations associated with the graph and, in particular, on the notion of a standard filtration. The main tool is the so-called internal metric that we introduce on simplices of measures; it is an iterated Kantorovich metric, and the key result is that the relative compactness in this metric guarantees a constructive enumeration of the ergodic invariant measures. Applications include a number of classical theorems on invariant measures.
 \end{abstract}

 \tableofcontents

\section{Introduction}
The goal of this paper is an exact formulation of a new method of finding ergodic measures on a Markov compactum or on the path space of  a Bratteli diagram invariant with respect to the tail equivalence relation. The method presupposes a certain uniform compactness with respect to the metric we introduce, and if this condition is satisfied, then the complete list of ergodic invariant measures is parametrized by the points of the compactification. The ergodicity of a measure in this case means that it satisfies the 0--1 law with respect to the tail filtration. The suggested method is an elaboration and strengthening of the ergodic method for finding the ergodic invariant measures, which is based on the ergodic theorem or the theorem on convergence of martingales and does not require other conditions. As a result, it can be applied to any approximation of invariant measures. But it is this fact that causes serious difficulties when calculating finite-dimensional distributions. The main difficulty in applying the ergodic method is not so much in calculations as in the proof that the obtained list is complete. But it turns out that in many natural situations, such as various generalizations of de Finetti's theorem, the theorem on the characters of the infinite symmetric group, etc., the natural approximation of the tail partition has much stronger properties than just convergence everywhere or in measure. Namely, the tail filtration turns out to be {\it standard} in the sense of the theory of filtrations, which means, in particular, that the approximation converges in a very strong sense. In other words, this can be expressed as follows: in many examples invariant measures possess latent generalizations of the independence property, and the standardness exactly takes this  adequately into account, being such a property itself. In fact, the present paper is a result of applying the theory of filtrations, on which the author has been working since the 1960s, to problems close to the representation theory of locally finite groups, their characters and combinatorics. Nevertheless, the presentation below is self-contained and does not require using other papers. Here we do not consider the most general graphs, restricting ourselves to what is needed for the most important examples. A more general class of examples, requiring additional definitions, will be considered in a paper which is now in preparation. We only note that the theory of filtrations allows one to look at finite approximations of invariant measures ``from infinity,'' i.\,e., from the viewpoint of the {\it decreasing sequence of the tail $\sigma$-fields on the path space, rather than that of the increasing sequence of the finite $\sigma$-fields of their beginnings}.

In Section~2 we give all necessary definitions and a survey of relevant known facts. The set of invariant (central) measures is a projective limit of simplices, and its vertices (the Choquet boundary) are the ergodic measures. That is why we pay attention to the geometry of projective limits, which is also useful in other problems. The main idea of the paper is described in Section~3; first we define the Kantorovich metric on a simplex in an elementary geometric form, and then define the so-called internal metric on a projective limit of simplices. This notion allows us to divide all ergodic measures into two classes: standard and nonstandard measures.\footnote{In the preliminary papers \cite{V13,V14} we used the term ``smooth'' and ``nonsmooth'' measures, graphs, etc., following the tradition of the theory of  $C^*$-algebras, but the terminology suggested here seems to be more appropriate.} The main theorem is given in Section~4: if the levels (pre-limit simplices) satisfy the uniform compactness condition, then the whole Choquet boundary consists of the standard ergodic measures and can be found as the completion of the pre-limit boundaries with respect to the internal metric; in other words, we obtain a parametrization of the set of ergodic measures and an interpretation of the parameters. Of course, the generality of this approach gives no hope for exact formulas for the distributions; they can be (and were) obtained using the specific features (analytic or probabilistic) in concrete cases. And though one can hardly hope for this in the general case, a revelation of the meaning of the parameters and a natural metric on them, and, above all, the interpretation of their independence, opens a new level in the understanding of problems concerning invariant measures. In this paper, for space considerations, we mention examples only briefly in Section~5, meaning to treat them in another paper.

\begin{figure}
\centering
\includegraphics[height=6cm]{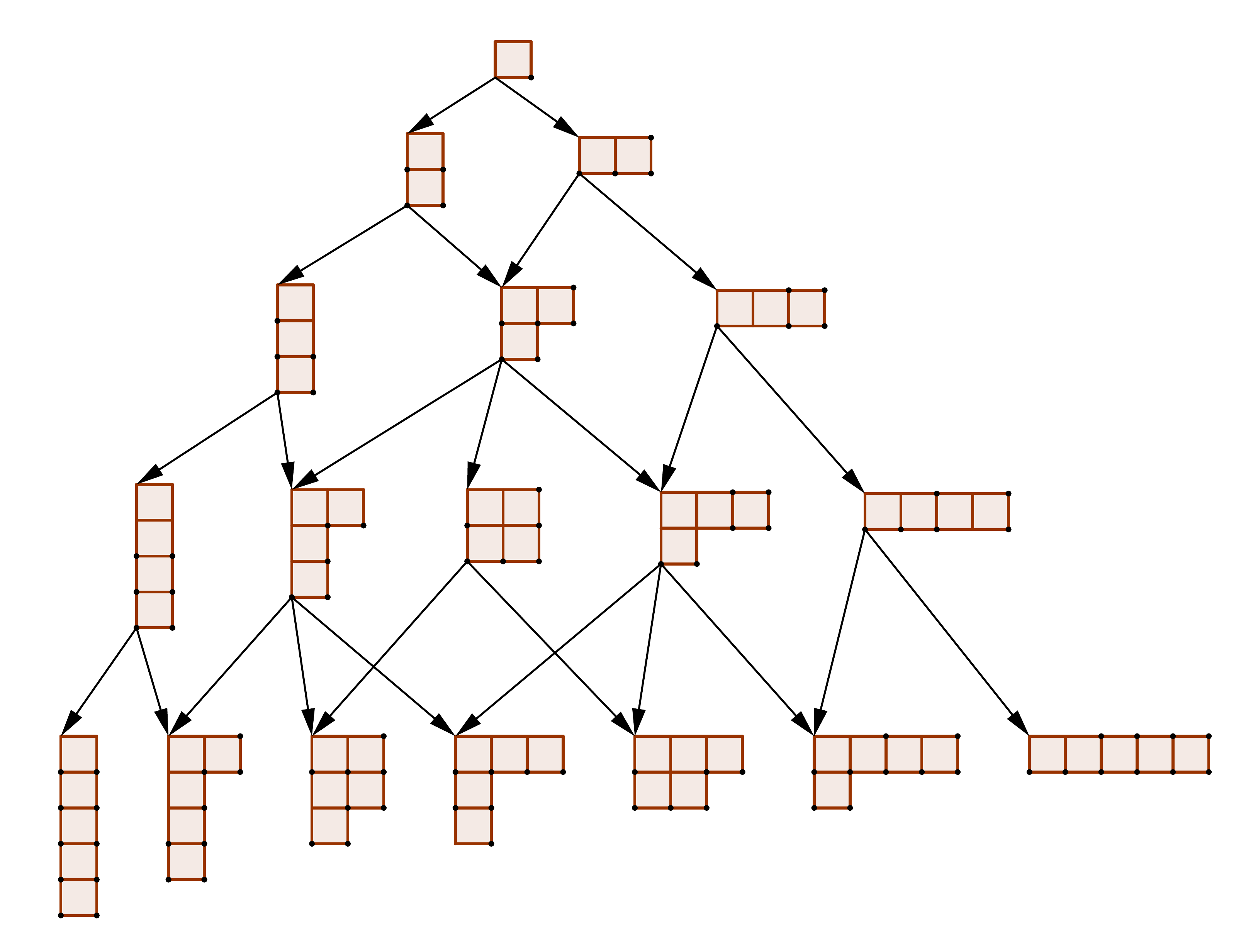}
\caption{The Young graph.}
\end{figure}

\begin{figure}
\centering
\includegraphics[height=6cm]{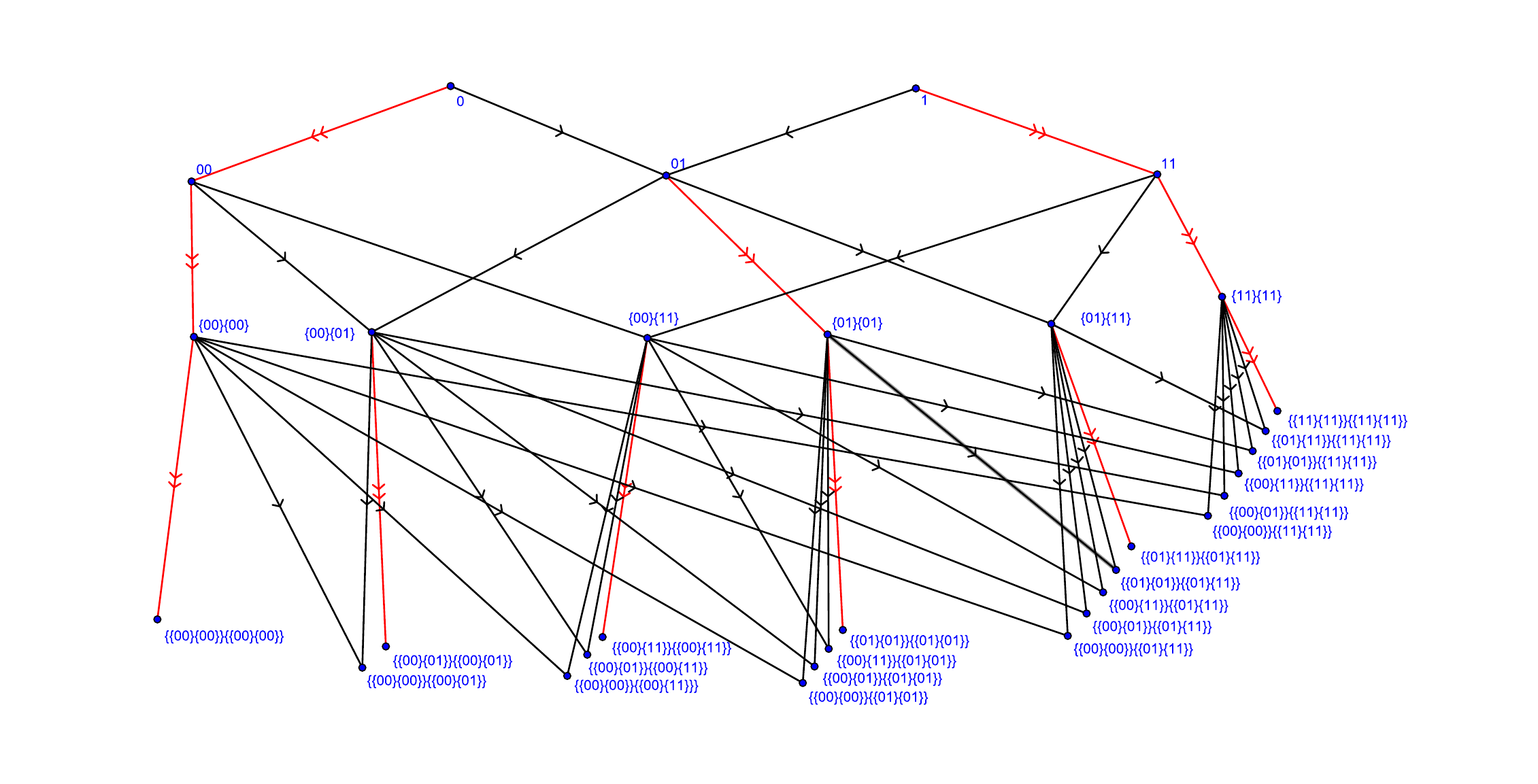}
\caption{The graph of unordered pairs.}
\end{figure}

In conclusion, we comment on the two figures. The first one shows the Young graph, i.\,e.,  the branching graph of irreducible representations of the infinite group of finite permutations. This graph is standard in the sense of this paper: all ergodic measures on its path space are standard, and they comprise the list of indecomposable characters of this group. But proving this and obtaining the complete list of measures is far from easy. First it was found analytically by E.~Thoma \cite{T}. Later, in  \cite{V-K}, the ergodic method idea was implemented, and the same list was obtained together with an explanation of the combinatorial and representation-theoretic meaning of the parameters, which was absent in \cite{T}; even more later, in  \cite{O} (see also \cite{Ol}), this result was obtained by methods which could be called operator-theoretic. The method of the present paper is intended to give the first purely combinatorial proof of the theorem on invariant measures (characters) that would use neither analytic nor algebraic techniques, but rely only on the combinatorics of the Young graph. The main point is that the completeness of the list, i.\,e., the proof that the frequencies uniquely determine an ergodic measure, a priori follows from the compactness theorem.
And the second graph is the graph of unordered pairs (the vertices of each level are all possible unordered pairs of vertices of the previous level), which is an example of a nonstandard graph. It illustrates in a simplest form the idea of a so-called ``tower of measures'' introduced in
\cite{94}.
Various problems concerning invariant measures for tame partitions (i.\,e., monotone limits of finite partitions) can be reduced to problems on graphs. Examples of nonstandard graphs often arise in  dynamical systems, in the theory of growth of random configurations, in statistical physics models,  etc.

\section{Basic notions}

\subsection{A Markov compactum and the path space of a graded graph}

We will consider the problem of describing the probability measures invariant with respect to an equivalence relation on the space of trajectories of a topological Markov compactum or on the space of paths of a graded graph. The equivalence relation will be given via an approximation by finite equivalence relations.

The basic space where the measures will be defined is a one-sided topological Markov compactum
${\cal X}$. Such a compactum is determined by the collection of state spaces, i.\,e., by a family of finite sets
$\{X_n\}$, $n=0,1, \dots$;  $X_0=\{\emptyset\}$, and a family of  $0-1$ matrices $M_n=\{\epsilon_{i,j}\}$, $i\in X_{n-1}$, $j\in X_n$, $n=1,2, \dots$, of orders $|X_{n-1}|\times |X_n|$, $n=1,2, \dots$, without zero rows and columns.
Elements of the Markov compactum ${\cal X}$ are arbitrary infinite sequences (trajectories)
$\{x_n\}_{n=0,1, \dots}$ of elements of the finite sets
$X_n$, $x_n \in X_n$, $n=1,2, \dots$, satisfying the admissibility condition which says exactly that every pair of
neighboring elements
 $x_{n-1}, x_n$, $n=1,2, \dots$,  satisfies $\epsilon_{x_{n-1},x_n}=1$.
 The topology and the $\sigma$-field of Borel sets in the space
${\cal X}=\{\{x_n\}_{n=0,1, \dots}\}$ are defined in a natural way, since
 $\cal X$ is, obviously, the inverse (projective) limit of the sequence of the finite spaces
${\cal X}_n$ that consist of the trajectories of length $n$ satisfying the same condition with respect to the projections of ``forgetting the last coordinate'': ${\cal X}_n \rightarrow {\cal X}_{n-1}$.

Clearly, the above data $\{X_n; M_n\}_{n\geq 0}$ determines also another object: an {\it infinite, locally finite,
$\Bbb N$-graded graph $\Gamma$}, i.\,e., a so-called Bratteli diagram. The Markov compactum defined above can be identified with the space of infinite paths in this graph, which will be denoted by
 $T(\Gamma) \sim \cal X$. Bratteli diagrams and their combinatorics are, in turn, very closely related to the theory of locally semisimple algebras and AF-algebras, but we will not enter into this here.

Thus we can interpret all that follows not only in terms of Markov compacta, which are natural for probabilistic statements, but also in terms of graded graphs $\Gamma$, their levels $\Gamma_n \sim X_n$, $n=0,1, \dots$, their vertices $ x_n \sim \gamma_n$, finite or infinite paths $\{t_n=x_n\}$, etc.
Since these two languages coincide almost tautologically (it suffices to rotate the picture of the graph by $90^{\circ}$ to obtain the scheme of the Markov chain), in what follows we use both languages simultaneously without fear of inconsistencies.\footnote{There is a trend in terminology, originated in the theory of the Young graph, in which vertices of the graph
$\Gamma$, i.\,e., states of the Markov chain, are called diagrams, paths are called tableaux, etc.; here we avoid this.} The language of graphs is used mainly for the combinatorial analysis of applications, which have been in fact the main source of the problems under discussion for the author. Moreover, it turned out that the arising examples, coming from the theory of Bratteli diagrams, are poorly studied  in the framework of Markov compacta, dynamical systems, and random processes; and, vice versa, useful constructions from the theory of dynamical systems only now start finding adequate applications in the algebraic theory.

\begin{remark}{\rm
It is convenient to adopt the following convention: no two distinct vertices of the same level of the graph (respectively, no two elements of the state space of the Markov compactum) have the same set of preceding vertices (i.\,e., the matrices $M_{n}$ have no identical rows). In what follows, this condition will ensure that the internal metric is nondegenerate. It is not too restrictive, since if there are such vertices, then we can replace them by a single vertex and change the multiplicities of edges by multiplying them with the number of merged vertices; applying this procedure to all vertices, we obtain a new graph satisfying the required condition.}
\end{remark}

The most important notions for what follows are the tail equivalence relation on a Markov compactum, or on the path space of a graph, and the tail filtrations on these spaces. {\it Two paths, or two sequences of elements in a Markov compactum, are called $n$-cofinal (equivalent) if they coincide from the $n$th level}; the $n$-cofinality equivalence classes are finite. The measurable (defined as the partition into the preimages of a Borel-measurable function) partition of the path space into these classes will be denoted by $\xi_n$. {\it Two paths are called cofinal if they are $n$-cofinal for some $n$.} The cofinality classes determine an equivalence relation which will be called the tail equivalence relation and denoted by  $\xi(\cal X)$ or $\xi(\Gamma)$.
In general, the partition into the cofinality classes, i.\,e., the intersection of the $n$-cofinality partitions
 $\xi_n$ over all $n$, is not measurable, but it is hyperfinite, i.\,e., by definition, is a limit of Borel equivalence relations with finite equivalence classes. The $\sigma$-field of $\xi_{n}$-measurable sets (i.\,e., sets that along with a given path contain all paths $\xi_{n}$-equivalent to it) will be denoted by ${\frak A}_n$. The $\sigma$-fields ${\frak A}_n$ decrease with $n$,
the first of them coinciding with the whole original $\sigma$-field: ${\frak A}_0=\frak A$. {\it Decreasing sequences of $\sigma$-fields in measure spaces or topological spaces are called filtrations; the above sequence
 $\{{\frak A}_n\}$ will be called the tail filtration on the path space of the graph or on the Markov compactum.} These notions and their properties, which have been studied since the 1970s (see \cite{V71, 94}),  play an important role in this paper.
\subsection{Additional structures on path spaces}

On the objects described above one can define important additional structures. We will consider them later, and now give only a short description.

\medskip
{\bf1.} The finite set of edges leading to a given vertex can be endowed with a linear ordering; if it is chosen for all vertices, this allows one to define a lexicographic ordering on the classes of cofinal paths, and then introduce a partial transformation of paths that sends a given path to the next one in the  lexicographic ordering, the so-called {\it adic shift}.
Here we will not use this notion; the only useful observation is that the tail equivalence relation is nothing but the trajectory partition for the adic shift.

\medskip
{\bf2.} Another additional structure is the cocycle of cotransition probabilities: with each vertex
 $\gamma_n\in \Gamma$ (or coordinate $x_n\in X_n$) one associates the probability vector
 $\{\lambda_{\gamma_n}^{\gamma_{n-1}}\}$ where $\gamma_{n-1}$ ranges over all vertices from
  $\Gamma_{n-1}$ that are connected with $\gamma_n$. This vector is interpreted as the vector of cotransition probabilities of the Markov chain, i.\,e., as  $\mbox{Prob}\{x_{n-1}|x_n\}$.

In this paper we will consider this structure only in the basic special case canonically associated with the graph and the Markov compactum. Namely,
{\it the cotransition probability   $\lambda_{\gamma_n}^{\gamma_{n-1}}$  (for a pair $(\gamma_{n-1} ,\gamma_n)$) is proportional to the number of paths leading to $\gamma_{n-1}$ from the initial vertex}.
Measures with such cotransition probabilities are called central (invariant), see below. However, the method suggested in this paper applies without changes to the case of general cotransition probabilities (i.\,e., to the theory of quasi-invariant measures). The number of paths leading from
 $\emptyset$ to a vertex $\gamma$ is usually denoted by $\dim \gamma$ (since this is the dimension of a certain module).

\medskip
{\bf3.} The matrices $M_n$ determining a Markov compactum may have nonnegative integer entries, which corresponds to the case of a multigraph (a graph with multiple edges); in this case, a path is a sequence of edges rather than vertices.
All  main theorems remain valid in this case. One may consider an even more general context, where multiplicities are nonnegative reals corresponding to weights of edges and paths.

\medskip
Note that all additional structures introduced above can be rephrased in terms of generalized matrices
 $M_n$ and interpreted in terms of locally semisimple algebras and AF-algebras associated with Bratteli diagrams; we will not enter into these questions here.

\subsection{Setting of the problem on invariant measures}

We turn to describing the main problem.

Assume that we are given a Markov compactum   $\cal X$ (or the path space $T(\Gamma)$ of a Bratteli diagram); the set  $\mbox{Meas}(\cal X)$ of all Borel probability measures on $\cal X$ is an affine compact (in the weak topology) simplex, whose extreme points are delta measures.
Since $\cal X$ is an inverse (projective) limit of finite spaces (namely, the spaces of finite paths), it obviously follows that
$\mbox{Meas}(\cal X)$ is also an inverse limit of finite-dimensional simplices
 $\hat \Sigma_n$, where $\hat \Sigma_n$ is the set of formal convex combinations of finite paths (or just the set of probability measures on these paths) leading from the initial vertex to vertices of level $n$, $n=1,2, \dots $, and the projections $\hat\pi_n: \hat\Sigma_n \to \hat \Sigma_{n-1}$ correspond to ``forgetting'' the last vertex of a path. Every measure is determined by its finite-dimensional projections to cylinder sets (i.\,e., is a so-called cylinder measure). We will be interested only in invariant (central) measures, which form a subset of $\mbox{Meas}(\cal X)$.

\begin{definition}
A Borel probability measure $\mu$ on a Markov compactum (= on the path space of a graph) is called central if for any vertex of an arbitrary level, the projection of this measure to the subfield of cylinder sets of finite paths ending at this vertex is the uniform measure on this (finite) set of paths.
\end{definition}

Other, equivalent, definitions of a central measure
 $\mu \in  \mbox{Meas}(\cal X)$ are as follows.

\medskip
{\bf1.} The conditional measure of $\mu$ obtained by fixing the ``tail'' of infinite paths passing through a given vertex, i.\,e., the conditional measure of $\mu$ on the elements of the partition  $\xi_n$, is the uniform measure on the initial segments of paths for any vertex.

\medskip
{\bf2.} The measure is invariant under any adic shift (for any choice of orderings on the edges).

\medskip
{\bf3.} The measure is invariant with respect to the tail equivalence relation.
\medskip

The term ``central measure'' stems from the fact that in the application to representation theory of algebras and groups, measures with these properties determine traces on algebras (respectively, characters on groups). In the theory of stationary (homogeneous) topological Markov chains, central measures are called measures of maximal entropy.

The set of central measures on a Markov compactum $\cal X$ (on the path space $T(\Gamma)$ of  a graph $\Gamma$) will be denoted by $\Inv({\cal X})$ or $\Inv (\Gamma)$.
Clearly, the central measures form a convex weakly closed subset of the simplex of all measures:
$\Inv({\cal X}) \subset \Meas(\cal X)$. The set  $\Inv({\cal X})$ of central measures is also a simplex, which can be naturally presented as a projective limit of the sequence of finite-dimensional simplices of convex combinations of uniform measures on the $n$-cofinality classes. In more detail:

\begin{statement}
The simplex of central measures can be written in the form
$$\Inv({\cal X})=\lim_{\leftarrow}(\Sigma_n; p_{n,m}),$$
or
$$ \Sigma_1 \leftarrow \Sigma_2 \leftarrow \dots\leftarrow \Sigma_n\leftarrow \Sigma_{n+1}\leftarrow
\dots  \leftarrow\Sigma_{\infty}\equiv \Inv({\cal X}),$$
where $\Sigma_n$ is the simplex of formal convex combinations of vertices of the $n$th level
$\Gamma_n$ (i.\,e., points of $X_n$), and the projection
$p_{n,n-1}: \Sigma_n \to \Sigma_{n-1}$ sends a vertex $\gamma_n \in \Gamma_n$ to the convex combination
$\sum \lambda_{\gamma_{n-1}}^{\gamma_n} \delta_{\gamma_{n-1}}\in \Sigma_{n-1}$ where the numbers $\lambda_{\gamma_{n-1}}^{\gamma_n}$ are uniquely determined by the condition that $\lambda_{\gamma_{n-1}}^{\gamma_n}$ is proportional to the number of paths leading from
$\emptyset$ to $\gamma_{n-1}$ (which is denoted, as already mentioned, by
$\dim\gamma_{n-1}$).\footnote{In the general (noncentral) case, the coefficients  $\lambda$ are the cotransition probabilities (see above).}  The general form of the projection is $p_{n,m}=\prod_{i=m}^{n+1} p_{i,i-1}$, $m>n$.
\end{statement}
\begin{proof}
The set of all Borel probability measures on the path space is a simplex which is a projective limit of the simplices generated by the spaces of finite paths of length $n$ in the graph, which follows from the fact that the path space itself is a projective limit with the obvious projections of ``forgetting'' the last edge of a path. The space of invariant measures is thus a weakly closed subset of this simplex, and we will show that it is also a projective limit of simplices (the fact that it is a simplex is well  known, see, e.g., \cite{F}). The projection $\mu_n$ of any invariant measure  $\mu$ to a finite cylinder of level $n$ is a measure invariant under changes of initial segments of paths and hence lies in the simplex defined above; since the projections preserve this invariance, $\{\mu_n\}$ is a point of the projective limit. It remains to observe that a measure is uniquely determined by its projections, which establishes a bijection between the points of the projective limit and the set $\Inv(\Gamma)$ of invariant measures.
\end{proof}

Recall that points of the simplex $\Sigma_n$ are probability measures on the points of $X_{n}$ (i.e., on the vertices of the $n$th level $\Gamma_{n}$), and the extreme points of $\Sigma_n$ are exactly these vertices. Remark~1 means that distinct vertices of the graph correspond to distinct vertices of the simplex.

Extreme points of the simplex $\Inv(\Gamma)$  of invariant measures on the whole path space $T(\Gamma)$
are indecomposable invariant measures, i.\,e., measures that cannot be written as nontrivial convex combinations of other invariant measures. Then it follows from the theorem on the decomposition of measures invariant with respect to a hyperfinite equivalence relation into ergodic components that an indecomposable measure is ergodic ( = there are no invariant subsets of intermediate measure). It is these measures that are of most interest to us, since the other measures are their convex combinations, possibly continual. The set of ergodic central measures of a Markov compactum
 $\cal X$ (of a graph~$\Gamma$) will be denoted by $\mbox{Erg}(\cal X)$ or $\mbox{Erg}(\Gamma)$.

\begin{problem}
Describe all central ergodic measures for a given Markov compactum (respectively, all indecomposable central measures for a given graph). A meaningful question is for what Markov compacta or graphs the set of ergodic central measures has an analytic description in terms of combinatorial characteristics of this compactum or graph, and what are these characteristics; and in which cases such a description does not exist. The role of such characteristics may be played by some properties of the sequence of matrices
 $\{M_n\}$ determining the compactum (graph), frequencies, spectra, etc.
\end{problem}

This problem includes those of describing unitary factor representations of finite type of discrete locally finite groups, finite traces of some $C^*$-algebras, Dynkin's entrance and exit boundaries (see \cite{Dy}); it is very closely related to the problems of finding Martin boundaries, Poisson--Furstenberg boundaries, etc. The answer to the question stated in Problem~1 may be either ``tame'' (there exists a Borel parametrization of the ergodic measures or the factor representations of finite type) or ``wild'' (such a parametrization does not exist). As is well known since the 1950s, in the representation theory such is the state of affairs in the theory of irreducible representations of groups and algebras. However, this also happens, though more rarely, in the theory of factor representations. But in many classical situations the answer in ``tame,'' which is a priori far from obvious.

For example, the characters of the infinite symmetric group, i.e., the invariant measures on the path state of the Young graph (see Fig.~1), have a nice parametrization, and this is a deep result; however, for the graph of unordered pairs (see Fig.~2) there is no nice parametrization. We emphasize that the presentation of  $\Inv(\Gamma)$ as a projective limit of simplices relies essentially on the approximation, i.\,e., on the structure of the Markov compactum (graph). Obviously, the answer to the stated question also depends on the approximation. The fact is that we can change the approximation without changing the stock of invariant measures, which is determined only by the tail equivalence relation. The dependence of our answers on the approximation will be discussed later (see the remark on the lacunary isomorphism theorem in the last section). But since in actual problems the approximation is explicit already in the setting of the problem, the answer should also be stated in its terms. See examples below.

\subsection{Geometric formulations}

We will recall some well known geometric formulations, since the language of convex geometry is convenient and illustrative in this context.

{\rm1.} The set of all Borel probability measures on a separable compact set invariant under the action of a countable group (or equivalence relation) is a simplex (= Choquet simplex, see \cite{F}), i.\,e.,\ a separable affine compact set in the weak topology whose any point has a unique decomposition into an integral with respect to a measure on the set of extreme points.\footnote{Choquet's theorem on the decomposition of points of a convex compact set into an integral with respect to a probability measure on the set of extreme points is a strengthening, not very difficult, of the previous fundamental Krein--Milman theorem saying that a convex affine compact set is the weak closure of the set of convex combinations of extreme points.}
The set of ergodic measures is the Choquet boundary, i.\,e., the set of extreme points, of this simplex; it is always a
$G_{\delta}$ set.
\medskip

{\rm2.} Terminology (somewhat less than perfect): a Choquet simplex is called a Poulsen simplex
\cite{Pou} if its Choquet boundary is weakly dense in it, and it is called a Bauer simplex if the boundary is closed. Cases intermediate between these two ones are possible.

\medskip
{\rm3.} A projective limit of simplices (see below) is a Poulsen simplex if and only if for any $n$ the union of the projections of the vertex sets of the simplices with greater numbers to the $n$th simplex is dense. The universality of a Poulsen simplex was observed and proved much later \cite{Li,Lu}:

\begin{statement}
All separable Poulsen simplices are topologically isomorphic as affine compacta; this unique, up to isomorphism, simplex is universal in the sense of model theory.\footnote{That is, every separable simplex can be mapped injectively into the Poulsen simplex, and an isomorphism of any two isomorphic faces of the Poulsen simplex can be extended to an automorphism of the whole simplex.}
\end{statement}

One can easily check that every projective limit of simplices arises when studying {\it quasi-invariant measures} on the path space of a graph, or Markov measures with given cotransition probabilities (see above). But in what follows we consider only central measures, i.\,e., take a quite special system of projections in the definition of a projective limit. However, there is no significant difference in the method of investigating the general case compared with the case of central measures. We will return to this question elsewhere.

We add another two simple facts, which follow from definitions.

\medskip
{\rm4.} Every ergodic central measure on a Markov compactum (on the path space of a graph) is a Markov measure with respect to the structure of the Markov compactum (the ergodicity condition is indispensable here).

\medskip
{\rm5.} The tail filtration is semi-homogeneous with respect to every ergodic central measure, which means exactly that almost all conditional measures for every partition $\xi_n$, $n=1,2, \dots$, are uniform.

The metric theory of semi-homogeneous filtrations will be treated in a separate paper.

\subsection{Extremality of points of a projective limit, and ergodicity of Markov measures}

We give a criterion for the ergodicity of a measure in terms of general projective limits of simplices, in other words, a criterion for the extremality of a point of a projective limit of simplices.

Assume that we are given an arbitrary projective limit of simplices
$ \Sigma_1 \leftarrow \Sigma_2 \leftarrow \dots \leftarrow\Sigma_n\leftarrow \Sigma_{n+1}\leftarrow \dots\leftarrow \Sigma_{\infty}$
with affine projections $p_{n,n-1}:\Sigma_n\rightarrow \Sigma_{n-1}$, $n=1,2, \dots$ (the general projection
$p_{m,n}:\Sigma_m\rightarrow \Sigma_n$ is given above).

Consider an element $x_{\infty} \in \Sigma_{\infty}$ of the projective limit; it determines, and is determined by, the sequence of its projections $\{x_n\}_{n= 1,2, \dots}$, $x_n \in \Sigma_n$,  to the finite-dimensional simplices. Fix positive integers
$n<m$ and take the (unique) decomposition of the element $x_m$, regarded as a point of the simplex $\Sigma_m $, into a convex combination of its extreme vertices $e_i^m$:
$$x_m=\sum_i c_m^i \cdot e_i^m, \quad \sum_i c_m^i=1, \quad c_m^i\geq 0;$$
denote by $\mu_m=\{c_m^i\}_i$ the measure on the vertices of $\Sigma_m$ corresponding to this decomposition. Project this measure $\mu_m$ to the simplex $\Sigma_n$, $n<m$, and denote the obtained projection
by $\mu_m^n$; this is a {\it measure on $\Sigma_n$, and thus a random point of $\Sigma_n$}; note that this measure is not in general concentrated on the vertices of the simplex $\Sigma_n$.

\begin{statement}[Extremality of a point of a projective limit of simplices]
A point $x_{\infty}=\{x_n\}_n$ of the limit simplex $\Sigma_{\infty}$ is extreme if and only if the sequence of measures $\mu_n^m$ weakly converges, as $m \to \infty$, to the delta measure $\delta_{x_n}$ for all values of $n$:
\begin{eqnarray*}\mbox{for every } \epsilon>0, \mbox{ for every } n \mbox{ there exists }K=K_{\epsilon,n}\mbox{ such that }\\  \mu_n^m(V_{\epsilon}(\mu_n))>1-\epsilon\quad \mbox{for every }m>K,
\end{eqnarray*}
 where $V_\epsilon(\cdot)$ is the $\epsilon$-neighborhood of a point in the usual (for instance, Euclidean) topology.
\end{statement}

It suffices to use the continuity of the decomposition of an arbitrary point
$x_{\infty}$ into extreme points in the projective limit topology, and project this decomposition to the finite-dimensional simplices; then for extreme points, and only for them, the sequence of projections must converge to a delta measure.

One can easily rephrase this criterion for our case $\Sigma_{\infty}=\Inv(\Gamma)=\Inv(\cal X)$.
Now it is convenient to regard the coordinates (projections) of a central measure
$\mu_{\infty}$ not as points of finite-dimensional simplices, but as measures  $\{\mu_n\}_n$ on their vertices
 (which is, of course, the same thing). Then the measures $\mu_m^n$
 should be regarded as measures on probability vectors indexed by the vertices of the simplex, and the measure
$\mu$ on the Markov compactum $\cal X$ (or on $T(\Gamma)$), as a point of the limit simplex $\Inv$. The criterion then says that $\mu$ is an ergodic measure (i.\,e., an extreme point of $\Inv$) if and only if the sequence of measures
$\mu_m^n$ (on the set of probability measures on the vertices of the simplex $\Sigma_n$) weakly converges as $m \to \infty$
to the measure $\mu_n$ (regarded as a measure on the vertices of $\Sigma_n$) for all $n$.

In probabilistic terms, our assertion is a topological version of the theorem on convergence of martingales in measure and has a very simple form: for every $n$, the conditional distribution of the coordinate $x_{n}$ given that the coordinate
 $x_m$, $m>n$, is fixed converges in probability to the unconditional distribution of  $x_n$ as $m\to \infty$.

According to this proposition, in order to find the finite-dimensional projections of ergodic measures, one should enumerate all delta measures that are weak limits of measures $\mu_n^m$ as $m\to \infty$. But, of course, this method is inefficient and tautological. The more efficient {\it ergodic method} requires, in order to be justified, a strengthening of this proposition, namely, replacing convergence in measure with convergence almost everywhere, i.\,e., the individual ergodic theorem, or pointwise convergence of martingales (see \cite{V74,V-K}). The point of this paper is the enhancement of the ergodic method via a new type of convergence of martingales.

\section{The internal metric and the root norm on projective limits of simplices}

\subsection{Metrics on finite sets and the root norm}

We begin with the following simple remark, which however plays a fundamental role. Namely, we give an elementary definition of the Kantorovich metric in the simplest finite-dimensional case in a purely geometric framework. Consider a finite set
$E=\{e_i\}$ and construct the finite-dimensional simplex $\Sigma(E)$ of formal convex combinations of points of this set. Assume that on the set $E$, i.\,e., on the set of vertices
$\{e_i\}$ of the simplex $\Sigma(E)$, there is a metric $\rho$. This metric can be extended to each of the edges of $\Sigma(E)$ as the Euclidean metric such that the length of the edge is equal to the distance between its endpoints. Is there
a natural way to extend it to  higher-dimensional faces and to the whole simplex, and what is the stock of such extensions?

\begin{theorem}[Definition of the extension of a metric]
There exists a functorial extension  $\bar \rho$ of a metric $\rho$ defined on the vertex set of a simplex to the whole simplex. We mean that the extension procedure described below is a functor from the category of finite metric spaces to the category of affine simplices of a real vector space with a metric.

This extension is nothing but the Kantorovich (or transportation) metric on the simplex; it determines a norm in the affine hull of the simplex (the Kantorovich--Rubinshtein norm), which is maximal in the class of all possible extensions of the metric from the set of vertices to the affine hull.
\end{theorem}

 \begin{proof}
The extension of the metric and the norm are defined as follows.

With every point $x\in \Sigma$ of the simplex we can associate the unique probability measure  $\nu_x$ on the set of vertices  $e_i$, $i=1,2,\dots, n$, of $\Sigma$ whose barycenter coincides with $x$:
$$x=\sum_i c_ie_i;  \quad \nu_x=\{c_i \delta_{e_i}\}; \quad c_i\geq 0,\quad \sum_i c_i=1.$$
Consider the Kantorovich distance $k_{\rho}$ between two measures $\nu_x$ and $\nu_y$ corresponding to points $x$ and
$y$ of the simplex; it is declared to be the extension $\bar\rho$ of the metric $\rho$ to the whole simplex. We give the definition of the Kantorovich metric in our notation:
$$\bar\rho (x, y) \equiv k_\rho(\nu_x, \nu_y)=\min_{\psi}\left\{\sum_{i,j} \psi_{i,j}\rho(e_i,e_j): \psi=\{\psi_{i,j}\}, \sum_j \psi_{ij}=c_i, \sum_i \psi_{ij}=d_j\right\}.$$

One can easily see that on the vertex set of $\Sigma$, the metric $\bar \rho$ coincides with the original metric: $\bar\rho(e_i,e_j)=\rho(e_i,e_j)$. On the edges of $\Sigma$, it coincides with the usual Euclidean metric with given distances between the endpoints of edges. However, on higher-dimensional faces, it depends substantially on the original metric on the vertices. It is convenient to embed the simplex into its affine hull and choose the sum of its vertices as the origin of the linear space.
Then the metric on the simplex determines a (Kantorovich--Rubinshtein) norm on the affine hull, which is given by the formula
$$\|c-c'\|=k_r(c,c'), \quad \|\lambda c\|= |\lambda|\|c\|,$$
where $c,c'$ are two points of the cone spanned by the simplex, and this norm is well defined since it does not depend on the representation of an element as the difference of points of the cone.

The maximality of this norm in the class of all norms on the affine hull for which $\|\delta_x - \delta_y\|=\rho(x,y)$
follows from the fact that the unit ball in this norm is the convex hull of the elements of the form
$\delta_x -\delta_y$ (the simple roots, see below).\footnote{Maximality is not the only property of the Kantorovich metric that singles it out from the set of all possible extensions of a metric from the vertex set of a simplex to the whole simplex; another such property is used in Proposition~5.}
 \end{proof}

Thus there is a finitely parametrized family of norms in $\mathbb{R}^n$ each of which is determined by its values on the differences of coordinate vectors. This  is exactly the family of Kantorovich metrics and Kantorovich--Rubinshtein norms in the {\it simplest (finite) case of the transportation problem}. The Kantorovich--Rubinshtein norm and its dual (Lipschitz) norm have been studied in many papers in the infinite-dimensional, continual, case; however, the geometry of balls in the above norms (i.\,e., in the finite-dimensional spaces $\mathbb{R}^n$) is not less interesting than the geometry of balls and norms in $l^{p}$ or other traditional norms. Strange as it may seem,
these norms are not mentioned in textbooks on optimization theory or convex geometry, as far as the author knows. In this relation, see the paper
 \cite{VPM} devoted to an entirely different subject.

The unit balls in these metrics are very interesting convex centrally symmetric polyhedra. In the dimension
$d=2$, this ball is an arbitrary  centrally symmetric hexagon, i.e., the norm is hexagonal; for  $d=3$, it is a polyhedron combinatorially equivalent to a cube. The general case is more complicated and, apparently, has not been studied.
{\it The metric $\bar \rho$ on the simplex $\Sigma$ can also be called the root metric corresponding to the metric $\rho$
on the vertices of $\Sigma$}. If $\rho=\delta_{x,y}=1$ for $x\neq y $, this is the root metric in the sense of the Lie algebra
 ${\widehat A}_n$.\footnote{This norm is related to the theory of Lie algebras in the following way. Assume that our simplex lies in the Cartan subalgebra of the Lie algebra ${\hat A}_n$ and is generated by its simple roots, on which the delta metric is considered. Then the unit ball in this norm is the convex hull of all positive and negative simple roots; it is natural to call it the {\it root norm}. The root norm is invariant under the Weyl group; has it been considered in detail (see \cite{Vin})?}

\begin{remark}{\rm
The definition immediately implies an important monotonicity property of the metrics under consideration:
$$\bar\rho_{n-1} (p_{n,n-1}x, p_{n,n-1}y)\leq \bar \rho_n (x,y), \quad x,y \in \Sigma_n.$$
}\end{remark}

\subsection{Definition of the internal metric}

We pass to an inductive construction of our key notion, that of the {\it internal metric} on projective limits of simplices and, in particular, on the simplex of invariant measures of a Markov chain. For this we iterate the Kantorovich metric with growing dimension. Assume that we are given a projective limit  $\lim_n \{\Sigma_n, p_{n,m}\}$ of simplices.

\begin{definition}
Fix an arbitrary metric $\rho_k$ on the vertex set of  the simplex $\Sigma_k$ for some $k\geq 1$
and its extension to the whole simplex $\Sigma_k$, which we will now denote by the same symbol
$\rho_k$.  We define metrics on all the subsequent simplices $\Sigma_n$, $n\geq k$, by induction as follows. By the definition of a projective limit, the vertices of  $\Sigma_{n+1}$, $n\geq k$, are projected to some points of the previous simplex
$\Sigma_n$, on which the metric is already defined. Take it as a metric on the vertices of
 $\Sigma_{n+1}$ and extend it as described above to a metric on the whole simplex $\Sigma_{n+1}$, which will also be denoted by $\rho_{n+1}$.
\end{definition}

In what follows, we assume that the original metric $\rho_k$ is fixed and denote the iterated internal metrics by
$\{\rho_n\}_{n=1}^{\infty}$; below we will show that the key property (standardness) does not depend on the choice of this metric.\footnote{Digressing for a moment from the main line, observe that the internal metric defined above has the following interpretation: it is a metric on the space of graded finite-dimensional modules, since this is exactly what the vertices of a Bratteli diagram are. This interpretation is of importance for algebraic applications. It would be interesting to compare it with other metrics on the spaces of modules of algebras if they exist.}

We have obtained a family of metrics, nondecreasing with $n$, on each of the simplices and, in particular, on the collection of all their vertices. We will show how to extend them to the disjoint union of the vertices of all simplices. First we define the distance between a vertex $x \in \Sigma_n$ and a vertex $f\in\Sigma_{n+1}$
as the distance between $x$ and the projection of $f$ to $\Sigma_n$   in the internal metric on $\Sigma_n$. Further,
the distance between vertices $e \in \Sigma_n$ and $f \in \Sigma_m$, $m>n$, is defined as the minimum of the sums of the distances $\rho(e_i,e_{i+1})$ over all pairs of vertices of neighboring levels for all chains of vertices $e=e_n,e_{n+1}, \dots, e_{m-1},e_m=f$ of intermediate levels with final point $f$ and initial point $e$. One could also define a metric on the whole union of the simplices, but we will not need it.

The definition of the internal metric on the vertices of a simplex allows us to translate it to the vertices of a graph and to the states of a Markov chain; by Remark~1, every vertex of the simplex is associated with only one vertex of the graph, hence the image of the metric on the vertices of the graph is a metric (rather than a semimetric).

\begin{remark}{\rm It is natural to give a general definition of sequences of internal metrics (or semimetrics) not for the vertices of simplices or graphs, but for the space of sequences of vertices (edges) of a graph, i.\,e., for the path space of the graph or the trajectory space of the Markov chain. This can be done by the universal trick of ``translating a semimetric'' similar to that described above. The iterations will result in a semimetric on the space of infinite paths, which can be quotiened to obtain a metric on the space of central measures, or, more exactly, on the space of disjoint Borel (but, in general, nonclosed) supports of ergodic central measures. Since in many applications it suffices to regard internal metrics as already defined and compactify the union of the vertex sets, we postpone the general definition till another paper.}
\end{remark}

\subsection{Independence on the original metric}

The fact that the standardness of an extreme point does not depend on the original metric used to construct the internal metric is implied by the following result.

\begin{statement}
Let $S$ be an arbitrary finite set, and let $\rho, \rho'$ be two metrics on $S$ such that
$$\rho (x,y)\leq r\cdot \rho'(x,y)$$ for all $x,y \in S$, where $r$ is a positive constant. Then the corresponding Kantorovich metrics on the simplex
$\Sigma(S)$ of measures satisfy the similar inequality: $$r_{\rho}(\mu_1, \mu_2) \leq r\cdot r_{\rho'}(\mu_1, \mu_2).$$
\end{statement}

\begin{proof}
We make use of the duality theorem that presents the value of the Kantorovich metric as the supremum over the unit ball of the integrals of Lipschitz functions with respect to the difference of the measures:
$$r_{\rho}(\mu_1, \mu_2)=\sup_{u \in \mbox{Lip}_{\rho,1}} \int_S u(x)d(\mu_1-\mu_2).$$
Obviously, the inequality for the metrics implies the following relation between the unit balls of the spaces of Lipschitz functions:
$$\mbox{Lip}_{\rho,1} \subset r\cdot \mbox{Lip}_{\rho',1},$$
and the required inequality follows.
\end{proof}

\begin{corollary}
The asymptotic (with respect to the dimension) properties of the internal metrics, such as convergence of sequences etc., do not depend on the original metric.
\end{corollary}

Indeed, by the proposition, inequalities between metrics are preserved under iterations.

In the general case, the original metric is defined on the space of paths (finite or not) starting at the zero-level vertex, and successive iterations translate it to paths starting from the first, second, etc.\ level (see Remark~3). The conclusion and the proof that the limit behavior of these metrics does not depend on the original one remain valid in this case, too.

\section{Standardness and uniform compactness}

\subsection{Standard extreme points and standard ergodic measures}

A trivial extreme point of a projective limit of simplices is a sequence  $\{x_n\}$ consisting of extreme points of the pre-limit simplices: $x_n\in \mbox{Ex}(\Sigma_n)$; such points are of no interest even if they exist. On the other hand, some extreme points of projective limits satisfy much stronger concentration conditions than in the extremality condition (Proposition~4).
The case of importance for us is where the sequence of projections $\{x_n\}$ is such that their distances to the set of vertices of the simplex in some metric tend to zero. In order to state exactly what such a convergence means, we must introduce a metric on the simplices of all dimensions according to some common rule.\footnote{This convergence can be universal if the definition is functorial.} Just such a definition was given above: this is the internal metric.

This allows us to divide the set of extreme points, and, in particular, the set of ergodic central measures on a Markov compactum and on the path space of a graph, into two classes.

Because of the importance of this definition, we first state it separately for an arbitrary projective limit of simplices, and then specialize for measures on Markov compacta and  on the path spaces of graphs. It is not difficult to see that these statements express the same property in different terms.\footnote{In this paper, we do not enter into the question of how all these formulations are related to the general definition of a standard filtration, neither to the question why the standardness in the sense of these definitions is a property of a filtration that does not depend on the way in which it is realized as a tail filtration. These questions will be studied in the above-mentioned paper in preparation on the metric and combinatorial theory of filtrations.}

\begin{definition}
{\rm1.} Consider a projective limit of simplices $\lim \{\Sigma_n, p_{m,n}\}\equiv \Sigma_{\infty}$. An extreme point
$x_{\infty}\in \Sigma_{\infty}$ with projections $\{x_n\in \Sigma_n\}$ is called standard if
$\lim_{n\to\infty} \rho_n(x_n, \mbox{\rm Ex}(\Sigma_n))=0$, i.\,e., the distance from $x_{n}$ to the boundary of the simplex in the internal metric tends to zero. Note that the extremality of the point is already implied by this condition. This property does not depend on the choice of the initial metric.

{\rm2.} A central measure $\mu$ on the path space of a graph $\Gamma$ is called standard if there exists a sequence of vertices
$\gamma_n\in\Gamma_n$ such that for every $\epsilon >0$ we have $E\lim \mu_n(V_{\epsilon}(\gamma_n)) =1$, where $V_{\epsilon}(\cdot)$ is a neighborhood in the metric $\rho_n$. This property may be called the asymptotic concentration of measure on paths, or the law of large numbers; it does not depend on the choice of the initial metric.

{\rm3.}  A measure of maximal entropy (central measure) on a Markov compactum $\cal X$ is called standard if it has the following property: the distance in the metric $\rho_n$ between the conditional distribution on
$X_n$ (given a fixed condition on $X_{n+1}$) tends to zero in measure. This property is a strengthening of the theorem on convergence of martingales, since the convergence of measures in the internal metric means the uniformness of the ordinary convergence in the dimension of finite-dimensional distributions; it does not depend on the choice of the initial metric.

Extreme points of a projective limit of simplices ( = ergodic central measures) that are not standard are called nonstandard.

A projective limit of simplices (respectively, a graph, a Markov compactum) is called standard if its all extreme points (ergodic measures) are standard.
 \end{definition}

In the theory of filtrations \cite{V71,94}, the notion of standardness was introduced as a nonobvious generalization of the independence property; for instance, a homogeneous ergodic filtration is the filtration of the $\sigma$-fields of pasts of a sequence of i.\,i.\,d.\  variables (Bernoulli scheme) if and only if it is standard (the standardness criterion from
 \cite{V71}). In the general homogeneous, and even semi-homogeneous, case, there are many nonisomorphic standard filtrations which are quite complicated. But a common property, which makes them akin to independence, is that for all such filtrations there is an ``additional'' basis, i.\,e., a well-structured, with respect to the filtration, family of measurable sets generating the whole $\sigma$-field. In the theory of stationary processes, the standardness of the filtration of pasts is a strengthening of the regularity (or Kolmogorov) property, i.\,e., the triviality of the intersection of the filtrations. It is natural to compare it with Ornstein's very weak Bernoulli property \cite{O}; these two properties are in general position, but akin in the terminology of formulations, the difference being in the metrics used on the spaces of conditional measures.

The standardness condition itself may also be called the {\it concentration in the internal metric}, since the
point is that the projections of the measure are concentrated near some point of the boundary in this metric, or that the paths of the graph for large $n$ lie in a small $\rho$-neighborhood of a single vertex. This property is sometimes called ``the existence of a limit shape,'' but the precise meaning of these words depends on the metric. Finally, this property may, with good reason, be called the generalized law of large numbers, which is especially clear from the second statement. However, one should keep in mind the role of the metric in limit-shape-like theorems.

We have singled out the class of projective limits (graphs, Markov compacta) for which all ergodic central measures are standard, and we arrive at the following fundamental problem.

\begin{problem}
Describe the standard projective limits of simplices (standard graphs, standard Markov compacta), e.g., in terms of their definition via the matrices $\{M_n\}$ or in other terms. The asymptotic nature of this problem (in the sense that the answer does not change if we change finitely many initial matrices) allows one to hope only for sufficient conditions for standardness.
 \end{problem}

Below we will see that a number of classical examples (the Pascal graph, the Young graph, random walks on groups, etc.) turn out to be standard. At the same time there exist nonstandard examples, and this is the general case. The suggested procedure allows one not only to establish the standardness when it holds, but also to obtain a parametrization of the ergodic measures.

\begin{definition}
A sequence $x_n\in\Sigma_{n}$ of vertices of the simplex $\Sigma$ (or a sequence $x_n \in \Gamma_n$ of vertices of the graph $\Gamma$) is called regular if there exists a standard extreme point  $x_{\infty}$ of the projective limit (respectively, a standard central measure on the path state of the graph) for which the sequence of projections to the simplices
  $\Sigma_{n}$ converges with the sequence $x_{n}$ in the internal metric.
\end{definition}

Regular sequences break into classes parametrized by all standard extreme points. One can easily check that the following regularity criterion holds.

\begin{statement}
A sequence $\{x_n\}$ is regular if it is a Cauchy sequence in the internal metric (see the definition at the end of Section~3.2).
\end{statement}

\subsection{Uniform compactness}

We begin with a simple lemma (proved as early as in \cite{94}).

 \begin{lemma}
If the number of vertices at the level $\Gamma_{n}$ of a graph $\Gamma$ (respectively, the number of points in the state space $X_{n}$ of a Markov compactum) for all $n$
is bounded by a constant not depending on $n$, then the graph (Markov compactum) is standard, i.\,e., all its central ergodic measures are standard.
 \end{lemma}

 \begin{proof}
 Note that the mean distance between points with respect to the iterations of the metric $\rho_n$ tends to zero, since the number of points does not increase and for at least one pair of points at the next step the distance will decrease with a constant factor; moreover, no distance can remain constant by the ergodicity, and the convergence to zero is uniform over all measures on the given compactum.
 \end{proof}

It is interesting what is the minimum growth rate of the number of vertices for which nonstandardness appears.

We will consider the internal metric on the levels $\Gamma_n$ of the graph, or on the state sets
$X_n$ of the Markov chain.

\begin{definition}
We say that the family $\Gamma_n$ is uniformly  compact in the internal metric if for every
 $\epsilon>0$ the number of points in an $\epsilon$-net for $\Gamma_n$ is uniformly bounded in $n$.
\end{definition}

If this condition is satisfied for the sets $\Gamma_n$, i.\,e., for the  vertex sets of the simplices, then it also holds for the family of the simplices themselves endowed with the internal metric.

\begin{corollary}
If there exists a nonstandard extreme point, then the family of the vertex sets is not uniformly compact.
\end{corollary}

Indeed, the uniform compactness would mean that for every $\epsilon>0$ there is a uniformly bounded set of points of
$\epsilon$-nets on all levels, and it would follow by Lemma~1 that every extreme point is standard.

The internal metric can be extended by continuity to the set of standard extreme points (standard ergodic measures).

\begin{theorem}
If the family of simplices $\{\Sigma_n\}_n$ (equivalently, the family of their vertices) is uniformly compact in the internal metric, then the projective limit of these simplices (respectively, the Bratteli diagram, the Markov compactum) is standard. In this case, the limit is a Bauer simplex (= the Choquet boundary is closed), and the internal topology on it coincides with the weak topology. The set of extreme points coincides with the set of limits of regular sequences.
\end{theorem}

\begin{proof}
By Corollary~2, it follows from the uniform compactness in the internal metric that every extreme point is standard. But every standard extreme point is the limit of a regular sequence of vertices. The completion of the space of (classes of) sequences with respect to the internal metric thus coincides with the set of all extreme points. Therefore, the Choquet boundary is compact. This also implies the coincidence of the topologies, since the identity map of the simplex endowed with the internal topology into the same simplex with the weak topology is, obviously, continuous.
\end{proof}

Briefly speaking, the uniform compactness of a family of simplices in the internal metric means the compactness of the set of extreme points and
the projective limit of simplices in the internal metric.

Note that the set of standard extreme points may be empty, but such a simplex itself may be or not be a Bauer simplex (i.\,e., have a weakly compact Choquet boundary). For instance, there are examples in which the Choquet boundary is closed, and even consists of a single point, but the extreme point is not standard and the simplices do not enjoy the required compactness in the internal topology.

\begin{conclusion}
We suggest a method for solving the problem of enumerating the central measures, which consists in establishing the uniform compactness of the family of the vertices of the simplices, or the vertices of the graph, or the states of the Markov compactum, in the internal metric. If the uniform compactness holds, then the completion of these spaces with respect to this metric is the set of all ergodic invariant measures. In this case, all these measures are standard and satisfy the concentration condition. An explicit parametrization of the set of ergodic measures follows from the construction. This conclusion remains true also for the problem of enumerating measures with given cotransitions (i.\,e., quasi-invariant measures), that is, for describing entrance and exit boundaries, Martin boundaries, etc.
\end{conclusion}

\section{Comments and examples}

\subsection{Dependence on the approximation and lacunary isomorphism}

Recall that, by the lacunary isomorphism theorem (for its homogeneous case, see \cite{94}),
every filtration is lacunary standard, i.\,e., contains a standard subsequence. It follows that if we allow to pass from a given graph (Markov compactum) to the graph (compactum) in which some (possibly, infinitely many) levels are omitted and the multiplicities of the new edges take into account the omitted vertices and edges, then a nonstandard central measure can become standard. We will call this operation a {\it rarefaction}. It changes the internal topology: in the rarefied graph the topology becomes weaker. Thus, in general, the stock of standard extreme points increases. Hence a rarefaction can turn a nonstandard graph into a standard one, if it makes all measures standard simultaneously.\footnote{It is not clear whether there exists a graph for which such a rarefaction is not possible; it is most likely that this will be the case for the universal or random graded graph.} However, in practice this operation is very inefficient, and obtaining a standard graph from a nonstandard one at this price is not reasonable. For example, for filtrations with positive entropy, the required rarefaction grows superexponentially. Moreover, by the essence of the question, it is natural to associate a parametrization of the central measures with the original graph rather than with a rarefied one.

\subsection{Examples of calculations}

Clearly, the practical conclusion from the above considerations is that one should calculate the internal metric on the levels and check whether the uniform compactness holds. This is a purely combinatorial problem, sometimes easy and sometimes not that easy.  In the compact case, the completion with respect to the internal metric provides an explicit parametrization of the extreme points.

\begin{theorem}
The Pascal graph of any dimension (${\Bbb Z}_+^{d+1}$) is standard. The list of ergodic invariant measures is the list of Bernoulli measures parametrized by the points of the unit
 $(d-1)$-dimensional simplex (for $d=2$, this is de Finetti's theorem). The Young graph (see Fig.~1) is also standard.
\end{theorem}

The Pascal graph of dimension $d$ is the lattice ${\Bbb Z}_+^d$ graded by the sum of the coordinates; for the ordinary Pascal graph,  $d=2$. The calculation of the internal metric in this case is outlined in \cite{V13}.  In this case, it suffices to use manipulations with binomial or multinomial coefficients. The proof of the standardness gives another proof of the generalized de Finetti theorem. We leave detailed calculations to a detailed paper about examples. A nontrivial fact here is the appearance of the root norm for the Cartan subalgebra of $\widehat{sl}(d+1,\mathbb{R})$ as the norm corresponding to the internal metric.

The standardness of the Young graph implicitly follows from Thoma's theorem (i.e., from the list of characters), but there is every reason to hope that one will be able to check the compactness in the internal metric directly   and thus obtain the first purely combinatorial proof of Thoma's theorem.
Besides formulas for the dimensions (e.g., the hook-length formula), an approximation of the graphs themselves may be useful: the uniform compactness in this case can be obtained from an approximation of the Young graph by multidimensional Pascal graphs of growing dimension. The Thoma parameter ranges over the simplex of two-sided nonnegative converging series with sum at most~$1$.

It is most likely that the standardness holds for a much more general case, which covers all the previous ones, namely, for the graphs that are the Hasse diagrams of distributive lattices with finitely many generators. In particular, the multidimensional Young graphs belong to this class. It is interesting that candidates for the role of parameters in all these cases are the {\it independent frequencies} of certain (minimal infinite) ideals, and the generalized independence announced above is exactly a certain independence of these frequencies. The standardness (i.\,e., the compactness of the set of standard extreme points) implies the compactness of the set of these frequencies. The reason behind the uniform compactness in all these cases (see above) lies in the following structure of the graphs: all vertices preceding any given vertex are close in the internal metric, and this closeness increases with the number of the level containing this vertex.

It is exactly the absence of this property that is a characteristic feature of the known examples of nonstandard graphs. For the graph of unordered pairs (see Fig.~2), one does not know the answers to simple questions related to the internal metric, for instance, what is the growth rate of the $\epsilon$-entropy. All these questions are closely related to the theory of so-called tower of measures \cite{94}.

\medskip
Translated by N.~V.~Tsilevich.

\end{document}